\documentclass[10pt, reqno]{amsart}
\usepackage{amssymb}
\usepackage{graphicx}
\usepackage{xcolor} % A package to add color.
\usepackage{fullpage} % Sets all margins to 1 in.  
\usepackage{amsmath}
\usepackage{amsthm}
\usepackage{verbatim}
\usepackage{enumitem}
\usepackage{amsaddr}
\DeclareFontFamily{U}{mathx}{\hyphenchar\font45}
\DeclareFontShape{U}{mathx}{m}{n}{
	<5> <6> <7> <8> <9> <10>
	<10.95> <12> <14.4> <17.28> <20.74> <24.88>
	mathx10
}{}
\DeclareSymbolFont{mathx}{U}{mathx}{m}{n}
\DeclareFontSubstitution{U}{mathx}{m}{n}
\DeclareMathAccent{\widecheck}{0}{mathx}{"71}
\DeclareMathAccent{\wideparen}{0}{mathx}{"75}

\usepackage{enumitem}
\setlist[enumerate]{leftmargin=1.5em}
\setlist[itemize]{leftmargin=1.5em}

\usepackage{seqsplit}

\newtheorem{theorem}{Theorem}
\newtheorem{corollary}{Corollary}[theorem]
\newtheorem{lemma}{Lemma}[section]

\newtheorem{remark}{Remark}[section]

\numberwithin{equation}{section}

\newcommand{\norm}[1]{{\left\Vert #1 \right\Vert}}
\newcommand{\normif}[1]{{\left\Vert #1 \right\Vert}_{L^{\infty}}}

\newcommand{\normb}[1]{{\left\Vert #1 \right\Vert}_{L^2}}

\newcommand{\RN}[1]{%
  \textup{\uppercase\expandafter{\romannumeral#1}}%
}

\allowdisplaybreaks

\begin{document}

\title{Finite-time blow-up in hyperbolic Keller-Segel system of consumption type with logarithmic sensitivity}

\author{Jungkyoung Na} \thanks{Department of Mathematics, Seoul National University. \newline  E-mail address: njkyoung09@snu.ac.kr}

\date{\today}

\maketitle

\begin{abstract}
This paper deals with finite-time blow-up of a hyperbolic Keller-Segel system of consumption type with the logarithmic sensitivity 
\begin{equation*}
    \partial_{t} \rho = -  \chi\nabla \cdot \left (\rho \nabla \log c\right),\quad \partial_{t} c = - \mu c\rho\quad (\chi,\,\mu>0)
\end{equation*}
in $\mathbb{R}^d\; (d \ge1)$ for nonvanishing initial data. 
This system is closely related to tumor angiogenesis, an important example of chemotaxis.
Our singularity formation is not because $c$ touches zero (which makes $\log c$ diverge)  but due to the blowup of $C^1\times C^2$-norm of $(\rho,c)$.   As a corollary, we also construct initial data near any constant equilibrium state which blows up in finite time for any $d\ge1$.    
\end{abstract}

\section{Introduction}
Chemotaxis refers to the motion of biological cells toward a higher (or lower) concentration of some chemical substance. Understanding and analyzing its mechanism is crucial since it describes many ubiquitous biological and ecological phenomena. After its first mathematical modeling by Patlak (\cite{P53}) and Keller-Segel (\cite{KS70, KS71}),  there have been many variations in their models to explain realistic phenomena more precisely.
One of the greatly important examples of chemotaxis is tumor angiogenesis, the new blood vessel formation induced by tumor cells. To be precise, 
by releasing vascular endothelial growth factor (VEGF), tumor cells induce endothelial cells to make new blood vessels toward them. In order to model the mechanism of tumor angiogenesis, the authors of \cite{LSN00} proposed the following system:
\begin{equation}\label{P-KS}
\left\{
\begin{aligned}
	&\partial_{t} \rho = \kappa\Delta \rho - \chi\nabla \cdot \left (\rho \nabla \log c\right), \\
	&\partial_{t} c = -\mu c \rho. \\
\end{aligned}
\right.
\end{equation}
 Here, $\rho$ denotes the density of endothelial cells, and $c$ represents the concentration of VEGF that is consumed by the cells. $\kappa>0$ denotes the diffusion coefficient of the cells, and $\chi,\,\mu>0$ represent the intensity of chemotaxis and the consumption rate of VEGF, respectively.

In this paper, we derive finite-time singularity formation of the following Cauchy problem for a hyperbolic counterpart of \eqref{P-KS}:
\begin{equation}\label{H-KS}
\left\{
\begin{aligned}
	&\partial_{t} \rho = - \chi \nabla \cdot \left (\rho \nabla \log c\right), \\
	&\partial_{t} c = -\mu c \rho, \\
	&\rho(x,0)=\rho_{0}(x), \; \, c(x,0)=c_{0}(x),
\end{aligned}
\right.
\end{equation}
for $\left (x, t\right) \in \mathbb{R}^d \times \left (0,T\right)$ with $d\ge1$. 
Our finite-time blow-up is not attributed to $c\rightarrow 0$, which makes $\log c \rightarrow \infty$, but due to divergence of $C^1\times C^2$-norm of $(\rho,c)$. Moreover, our result considers initial data not touching zero, where in the case of initial data vanishing at some point, finite-time blow-up can be shown using propagation of certain degeneracy (\cite{IJ21}).
To effectively discuss the motivation for dropping $\Delta \rho$ and the significant meanings of our results, we first review some related previous results.

\subsection{Previous works}\label{subsection previous works}
There have been a lot of studies on the long-time dynamics of \eqref{P-KS}. In addition to its biological significance, \eqref{P-KS} contains a mathematically interesting structure: logarithmic sensitivity, $\log c$. Since $\log c$ diverges at $c=0$, this sensitivity has given difficulties in analyzing \eqref{P-KS}.
To overcome these, the Cole-Hopf transformation
 \begin{equation}\label{C-H trans}
     q=-\frac{\nabla c}{c}=-\nabla \log c,
 \end{equation}
and scalings
\begin{equation*}\label{scaling}
    x\rightarrow \frac{\sqrt{\chi\mu}}{\kappa}x,\quad t\rightarrow \frac{\chi \mu}{\kappa}t, \quad
    q\rightarrow \sqrt{\frac{\chi}{\mu}}q
\end{equation*}
have been used to transform \eqref{P-KS} into
\begin{equation}\label{P-KS'}
    \left\{
    \begin{aligned}
        &\partial_t \rho =\Delta \rho+ \nabla \cdot (\rho q), \\
        &\partial_t q = \nabla \rho.
    \end{aligned}
    \right.
\end{equation}
We briefly review the main results of \eqref{P-KS'}.
Global existence of classical solutions to \eqref{P-KS'} in a one-dimensional bounded domain was obtained for small initial data in \cite{ZZ07}, and for large initial data in \cite{TWW13} and \cite{LZ15}. 
The paper \cite{LZ15} also showed that the solutions converge to their boundary data at an exponential rate as time goes to infinity. In the one-dimensional whole line $\mathbb{R}$, the authors of \cite{GXZZ09} established the global existence of classical solutions to \eqref{P-KS'} for large initial data. Furthermore, the authors of \cite{LPZ15} showed that the global classical solutions for large initial perturbations around constant equilibrium states in $\mathbb{R}$ converge to the equilibrium states as time approaches infinity. The authors of \cite{LW09} and \cite{LW10} proved the nonlinear stability of traveling wave solutions, while the authors of \cite{LWZ13} showed the stability of composite waves.
In multidimensional space, global existence and long-time behaviors of classical solutions to \eqref{P-KS'} were obtained for initial data near a constant equilibrium state. In two or three-dimensional bounded domain $\Omega$, the paper \cite{LPZ12} obtained the global existence and large-time asymptotic behavior of classical solutions for initial data $(\rho_0,q_0)$ satisfying $\norm{(\rho_0-avg(\rho_0),q_0)}_{(H^3 \times H^3)(\Omega)}\ll 1$, where $avg(\rho_0)$ denotes the average of $\rho_0$ over $\Omega$.
The authors of \cite{LLZ11} showed global well-posedness and long time behavior of classical solutions in $\mathbb{R}^d$ $(d\ge2)$ when $\norm{(\rho_0-\bar{\rho},q_0)}_{(H^s \times H^{s})(\mathbb{R}^d)}\ll 1$ with $s>\frac{d}{2}+1$ for some constant $\bar{\rho}>0$.
The paper \cite{H12} generalized these results to the critical Besov space $B^{\frac{d}{2}-2}(\mathbb{R}^d)\times(B^{\frac{d}{2}-2,\frac{d}{2}-1}(\mathbb{R}^d))^d$ with $d\ge2$ under the assumption that
$\norm{(\rho_0-\bar{\rho},q_0)}_{B^{\frac{d}{2}-2}\times(B^{\frac{d}{2}-2,\frac{d}{2}-1})^d}\ll 1$ is sufficiently small for some constant $\bar{\rho}>0$. In $\mathbb{R}^3$, the authors of \cite{DL14} established global well-posedness of classical solutions when $\norm{(\rho_0-\bar{\rho},q_0)}_{(L^2 \times H^1)(\mathbb{R}^3)} \ll 1$ for some constant $\bar{\rho}>0$. Moreover, for the initial data satisfying $\norm{(\rho_0-\bar{\rho},q_0)}_{(H^2 \times H^1)(\mathbb{R}^3)}\ll 1$, the paper \cite{DL14} further derived decay property of the solutions. 
Later, the authors of \cite{WXY16} obtained asymptotic decay rates of classical solutions in $\mathbb{R}^d$ with $d=2,3$ under assumptions that $(\rho_0-\bar{\rho}, q_0)\in ((H^{2}\cap \dot{H}^{-s})\times (H^{2}\cap \dot{H}^{-s}))(\mathbb{R}^d)$ for some $s\in(0,\frac{d}{2})$ and $\norm{(\rho_0-\bar{\rho},q_0)}_{H^1 \times H^1}\ll 1$.
Recently, there have been studies of long-time behaviors of multi-dimensional solutions to \eqref{P-KS'} with smallness assumptions only on the low frequency part of initial data. On the unit square or cube with various types of boundary conditions, the authors of \cite{RWWZ19} showed global existence of classical solutions and convergence to equilibrium states by only assuming the smallness of entropic energy. In $\mathbb{R}^d$ $(d=2,3)$, the authors of \cite{WWZ21} and \cite{LWWWZ21} established the global well-posedness
and long-time behavior results for initial data potentially having large
$L^2$-energy.

\medskip

We can compare \eqref{P-KS} with the following classical counterpart:
 \begin{equation}\label{cP-KS}
\left\{
\begin{aligned}
&\partial_{t} \rho = \kappa \Delta \rho - \chi\nabla \cdot \left (\rho \nabla \log c\right), \\
	&\partial_{t} c = \mu c^{\lambda} \rho,
\end{aligned}
\right.
\end{equation}
where the main difference from \eqref{P-KS} is that $\rho$ does not consume but produces $c$. Depending on the range of $\lambda$, finite-time blow-up or global well-posedness of \eqref{cP-KS} have been studied by many authors. Since our main concerns are focused on the consumption type, we only list some references in this direction. See \cite{LS97, LS01, YHL01} for finite-time singularity formation when $\lambda=1$, and \cite{YHL01} for global regularity when $\lambda=0$. Some blow-up criteria and asymptotic behaviors of solutions can be found in \cite{KSV10, JK14}, and references therein. Refer to \cite{HA97} for numerical investigation.

\medskip

The first attempt to drop $\Delta \rho$ was from \cite{C18}.
To be more precise, the author of \cite{C18} considered
\begin{equation}\label{dH-KS}
\left\{
\begin{aligned}
	&\partial_t \rho =\nabla \cdot (\rho q) + \rho(1-\rho), \\
        &\partial_t q = \nabla \rho,
\end{aligned}
\right.
\end{equation}
where a term $\rho(1-\rho)$ represents a logistic growth restriction. In other words, the author considered the case with damping rather than diffusion of $\rho$. In this system, the author of \cite{C18} obtained the existence of global smooth solutions in $\mathbb{R}^2$ when $\norm{(\rho_0-1,q_0)}_{(H^3\times H^3)(\mathbb{R}^2)}\ll 1$ and $\nabla \times q_0 =0$. Recently, in \cite{IJ21} (see also \cite{BG22}), the authors proved the finite-time singularity formation of \eqref{H-KS} in $\mathbb{R}^d$ in the case when initial data satisfies some vanishing condition. To be precise, they showed that for initial data $(\rho_0,c_0)$ satisfying
$\rho_0(x_0)=0$ and $\nabla \rho_{0}(x_0)= \nabla c_{0}(x_0)=\bold{0}$ for some $x_{0} \in \mathbb{R}^d$, there exists a time $T^*>0$ such that
\begin{equation}\label{blowup of injee}
    \lim_{t\rightarrow T^*} \left( \norm{\rho(\cdot,t)}_{W^{2,\infty}(\mathbb{R}^d)} + \norm{c(\cdot,t)}_{W^{2,\infty}(\mathbb{R}^d)} \right) =\infty.
\end{equation}

\subsection{Main results and Discussion}\label{subsection main results}
Our primary motivation for studying the hyperbolic version \eqref{H-KS} is that finite-time singularity formation of \eqref{H-KS} could lead to rapid norm growth of solutions to \eqref{P-KS} when the diffusion coefficient $\kappa$ is small. In fact, in \cite{LSN00} where \eqref{P-KS} was initially derived, very small $\kappa$ was used, and numerical simulations showed large norm growth.
This kind of result can be found in other hyperbolic-elliptic type Keller-Segel equations (see \cite{Winkler14, KKS16}). 

The previous works covered in the last subsection raise three interesting questions:
\begin{itemize}
    \item[\textbf{Q1.}] Can we delete the vanishing condition of initial data in \cite{IJ21}? In other words, is there any nonvanishing initial data leading to finite time singularity formation?
    \item[\textbf{Q2.}] If some nonvanishing initial data blows up at some finite time, then is it because $c$ touches zero at some point, which makes logarithmic sensitivity $\log c$ blow up? Furthermore, can we 
    get more detailed behaviors of solutions than \eqref{blowup of injee} at the blow-up time?
    \item[\textbf{Q3.}] Compared with global regularity results near constant equilibrium states for \eqref{P-KS} and \eqref{dH-KS}, what happens to solutions of \eqref{H-KS} near constant equilibrium states? 
\end{itemize}
To the best of the author's knowledge, there are no known answers to these questions.
The aim of this paper is to answer them. 
To begin with, our physical motivation for \textbf{Q1} is to consider non-vacuum states. 
Moreover, considering nonvanishing data, we can have an opportunity to deal with \textbf{Q3}. 
In order to explain how difficult it is to answer \textbf{Q1},
we need to compare \textbf{Q1} with vanishing conditions in \cite{IJ21}: $\rho_0(x_0)=0$, $\nabla \rho_{0}(x_0)= \nabla c_{0}(x_0)=\bold{0}$ for some $x_{0} \in \mathbb{R}^d$.
The authors of \cite{IJ21} employed propagation of these conditions to prove the finite-time blow-up. Heuristically, if initial cells don't exist at $x_0$ but initial chemicals are abundant at $x_0$, cells would be concentrated fast in the neighborhood of the point to consume the chemicals. But due to the propagation of the vanishing conditions, cell density remains zero at the point at which solutions blow up within some finite time. But this scenario does not work in nonvanishing initial data, so we need to consider other scenarios to deal with it.

To do so, in this note, we shall assume that initial data $(\rho_0,c_0)$ satisfies
\begin{equation}\label{asm-1}
    \left\{
    \begin{aligned}
        &(\rho_0, c_0) \in (L^{\infty}\times L^{\infty})(\mathbb{R}^d), \\
        &\rho_0(x) \ge \beta_1 >0, \quad c_0(x) \ge \beta_2 > 0 \quad \text{for some constants} \;\, \beta_1,\,\beta_2.
    \end{aligned}
    \right.
\end{equation}
Under this assumption, the dynamics of \eqref{H-KS} guarantees $\rho, c >0$, and $\normif{c(t)}\le \normif{c_0}<\infty$.

In order to investigate whether solutions to \eqref{H-KS} blow up at a finite time, we firstly have to establish local well-posedness of \eqref{H-KS} for nonvanishing initial data in $\mathbb{R}^d$ with $d\ge1$.
\begin{theorem}[Local well-posedness for nonvanishing data]\label{thm: lwp}
Let $d\ge1$. Assume that the initial data $(\rho_0,c_0)$ satisfies \eqref{asm-1}, and
\begin{align}\label{asm-lwp-1}
    \left(\nabla\rho_0,\nabla c_0\right)\in (H^{m-1} \times H^{m})(\mathbb{R}^d)
    \quad \text{for some}\;\,m>\frac{d}{2}+1.
\end{align}
Then there exist a time $T>0$, time-dependent positive constants $\beta_1(t)$, $\beta_2(t)$, and a unique solution $(\rho,c)$ to \eqref{H-KS} such that
\begin{equation}\label{asm-lwp-2}
    \left\{
    \begin{aligned}
        &\rho(x,t) \ge \beta_1(t), \,\; \quad c(x,t) \ge \beta_2(t) \quad \text{on} \;\,[0,T], \\
    &(\rho, c) \in L^{\infty}\left([0,T]; (L^{\infty} \times L^{\infty})(\mathbb{R}^d)\right),  \\
    &\left(\nabla \rho, \nabla c\right) \in L^{\infty}\left([0,T];(H^{m-1} \times H^{m})(\mathbb{R}^d)\right).
    \end{aligned}
    \right.
\end{equation}
Furthermore, if initial data satisfies $(\nabla\rho_0,\nabla c_0)\in (H^{\infty} \times H^{\infty})(\mathbb{R}^d)$, then the unique solution satisfies $\left(\nabla \rho, \nabla c\right) \in L^{\infty}\left([0,T];(H^{\infty} \times H^{\infty})(\mathbb{R}^d)\right)$, where
$H^{\infty}(\mathbb{R}^d):=\bigcap_{k\ge0}H^{k}(\mathbb{R}^d)$.
\end{theorem}
In order to show this theorem, we  compare $\nabla^m \rho$ with $\sqrt{\rho} \nabla^{m+1}\log c$ (rather than $\nabla^{m+1} \log c$) and observe a cancellation of certain quantity which involves highest derivatives. This idea was motivated by \cite{IJ21}. We shall prove Theorem \ref{thm: lwp} in Section \ref{pf: lwp}.

\medskip

Next, in Section \ref{pf: fb-1}, we present some sufficient conditions of nonvanishing initial data for finite-time singularity formation when $d=1$.
\begin{theorem}[Sufficient conditions of data for finite-time blow-up in $\mathbb{R}$]\label{thm: blow-1}
Suppose that initial data $\left (\rho_{0}, c_{0}\right)$ satisfies \eqref{asm-1} with $d=1$,
\begin{subequations}
    \begin{align}
        &\bullet\; (\partial_x\rho_0,\partial_x c_0) \in  (H^{\infty} \times H^{\infty})(\mathbb{R}), \label{asm-2}\\
    &\bullet\; \text{$\rho_{0}$, $c_{0}$, and $\partial_{x}c_0$ attain their maximums and minimums at some points in $\mathbb{R}$} \label{asm-3},\\
    &\bullet\; \text{$\partial_{x}\rho_{0}(x_{0})\le0$ and  $c_0(x_0)\partial_{xx}c_{0}(x_{0})<(\partial_{x}c_{0}(x_{0}))^2$ at some $x_{0} \in \mathbb{R}$.} \label{asm-4}
    \end{align}
\end{subequations}
Then there exists some finite time $T^*>0$ such that the unique solution $(\rho, c)$ to \eqref{H-KS} satisfies
\begin{equation}\label{fb-1-but bounded}
    \sup_{t\in[0,T^*)}\left({\left\Vert \rho(\cdot,t) \right\Vert}_{L^{\infty}(\mathbb{R})}  +{\left\Vert c(\cdot,t)  \right\Vert}_{W^{1,\infty}(\mathbb{R})} +  {\left\Vert \partial_x \log c(\cdot,t)  \right\Vert}_{L^{\infty}(\mathbb{R})}\right)< \infty
\end{equation}
and
\begin{equation}\label{fb-1-real bounded}
    \lim_{t\to T^{*}} \left({\left\Vert \partial_x\rho(\cdot,t)  \right\Vert}_{L^{\infty}(\mathbb{R})}  +{\left\Vert \partial_{xx} c(\cdot,t)  \right\Vert}_{L^{\infty}(\mathbb{R})}\right) = \infty.
\end{equation}
\end{theorem}
\begin{remark}\label{rmk. bump function}
    There is a large class of functions satisfying \eqref{asm-1} and \eqref{asm-2}-\eqref{asm-4}.
    For instance, let $\psi \in C^{\infty}_c(\mathbb{R})$ be a smooth bump function such that
    \begin{align*}
 \psi(x)=
     \begin{cases}
      1 \quad (|x| \le 1), \\
      0 \quad (|x|\ge 2).
     \end{cases}
 \end{align*}
 Then we can find $x_0 \in \mathbb{R}$ satisfying $\partial_x\psi(x_0)\le 0$ and $\partial_{xx}\psi(x_0) < 0$. Hence for any constants $\Bar{\rho},\,\Bar{c}>0$, if we define
    \begin{equation}\label{1D blowup data given in rmk1.1}
        \rho_0 := \Bar{\rho} +\psi \quad \text{and} \quad c_0 := \Bar{c} + \psi,
    \end{equation}
    then we can check that $(\rho_0,c_0)$ satisfies \eqref{asm-1} and \eqref{asm-2}-\eqref{asm-4}.
\end{remark}
Theorem \ref{thm: blow-1} answers our \textbf{Q1} and \textbf{Q2} in one dimension. Regarding \textbf{Q1}, nonvanishing initial data given in \eqref{1D blowup data given in rmk1.1} leads to the finite-time blow-up. Furthermore, for \textbf{Q2}, noticing $\sup_{t\in[0,T^*)}{\left\Vert \partial_x \log c(\cdot,t)  \right\Vert}_{L^{\infty}(\mathbb{R})}< \infty$ as we can see in \eqref{fb-1-but bounded}, we conclude that our singularity formation is not because $c$ touches zero but derivatives of $\rho$ and $c$ blow up as \eqref{fb-1-real bounded}. 
To the best of the author's knowledge, it is the first time to show that the finite-time blow-up of systems with logarithimic sensitivity is not attributed to the singularity from $c \rightarrow 0$. 

The main difficulty of showing Theorem \ref{thm: blow-1} lies in finding proper variables to control the equation for $\rho$ containing many derivatives. A key tool to solve this problem is a diffeomorphic transformation $w=(w_1,w_2):\Omega\rightarrow w(\Omega)$ with $ \Omega:=\left\{(z_{1},z_{2})\subset \mathbb{R}^2 : z_{1}>0 \right\}$. This $w$ is called Riemann invariants (refer to \cite{Lax} or \cite{Eva} for instance). In addition to the general definition of Riemann invariants, we  consider $w$ satisfying certain extra properties. 
Combining these extra properties with our assumptions \eqref{asm-1}, \eqref{asm-2}-\eqref{asm-4},
we deduce the stated finite-time singularity formation by using newly defined variables: $w_1(\rho,\log c)$, $ w_2(\rho,\log c)$.

\medskip

With the above one-dimensional result at hand, we finally prove the existence of a set of nonvanishing smooth initial data in $\mathbb{R}^d$ with $d\ge1$ making the solution singular in finite time. Our result is divided into $\mathbb{R}$ and $\mathbb{R}^d$ $(d\ge2)$ cases, respectively:
\begin{theorem}[Finite-time blow-up in $\mathbb{R}^d$]\label{thm: blow-d}
$\quad$

$\bullet\;$ Let $d=1$. Then for any interval $I\subset \mathbb{R}$, there exist some finite time $T^*>0$ and initial data $\left (\rho_{0}, c_{0}\right)$ satisfying \eqref{asm-1} and $(\nabla\rho_0,\nabla c_0)\in (H^{\infty} \times H^{\infty})(\mathbb{R})$ such that the unique solution $(\rho, c)$ to \eqref{H-KS} satisfies \eqref{fb-1-but bounded} and
\begin{equation*}
    \lim_{t\to T^{*}} \left({\left\Vert \partial_x\rho(\cdot,t)  \right\Vert}_{L^{\infty}(I)}  +{\left\Vert \partial_{xx} c(\cdot,t)  \right\Vert}_{L^{\infty}(I)}\right) = \infty.
\end{equation*}

$\bullet\;$ Let $d\ge2$. Then there exist some finite time $T^*>0$ and initial data $\left (\rho_{0}, c_{0}\right)$ satisfying \eqref{asm-1} and $(\nabla\rho_0,\nabla c_0)\in (H^{\infty} \times H^{\infty})(\mathbb{R}^d)$ such that the unique solution $(\rho, c)$ to \eqref{H-KS} satisfies
\begin{equation*}\label{result rd}
    \lim_{t\to T^{*}} \left({\left\Vert \rho(\cdot,t)  \right\Vert}_{W^{1,\infty}(\mathbb{R}^d)}  +{\left\Vert c(\cdot, t)  \right\Vert}_{W^{2,\infty}(\mathbb{R}^d)}\right) = \infty.
\end{equation*}
\end{theorem}
We need some remarks to interpret our results as well as to answer our questions \textbf{Q1} and \textbf{Q2}.
\begin{remark}\label{rmk: torus 1}
    Our result in $\mathbb{R}$ tells us that we can control the region where the finite-time blow-up occurs as we please. Hence we can obtain the same blow-up result in torus $\mathbb{T}:=\mathbb{R}/ \mathbb{Z}$ by making the singularity formed within the interval $(0,1)$ and extending the solution periodically.
\end{remark}
\begin{remark}\label{rmk: torus d}
    Concerning $\mathbb{R}^d$ with $d\ge2$, our result does not give as much information about blow-up  as the 1D case. This is because we cannot make use of Riemann invariants technique in the proof of  the multi-dimension case. However, noticing the previous remark, we can obtain the following result in torus $\mathbb{T}^d:=(\mathbb{R}/ \mathbb{Z})^d$ for any $d\ge 1$:

    \medskip
    
     There exist some finite time $T^*>0$ and initial data $(\rho_0,c_0)\in (C^{\infty} \times C^{\infty})(\mathbb{T}^d)$ such that the unique solution $(\rho, c)$ to \eqref{H-KS} satisfies
     \begin{equation*}
    \sup_{t\in[0,T^*)}\left({\left\Vert \rho(\cdot,t) \right\Vert}_{L^{\infty}(\mathbb{T}^d)}  +{\left\Vert c(\cdot,t)  \right\Vert}_{W^{1,\infty}(\mathbb{T}^d)} +  {\left\Vert \nabla \log c(\cdot,t)  \right\Vert}_{L^{\infty}(\mathbb{T}^d)}\right)< \infty
\end{equation*}
and
\begin{equation*}
    \lim_{t\to T^{*}} \left({\left\Vert \nabla \rho(\cdot,t)  \right\Vert}_{L^{\infty}(\mathbb{T}^d)}  +{\left\Vert \nabla^2 c(\cdot, t)  \right\Vert}_{L^{\infty}(\mathbb{T}^d)}\right) = \infty.
\end{equation*}

\medskip

Indeed, since Remark \ref{rmk: torus 1} guarantees a 1D solution $(\rho^1(x_1,t),c^1(x_1,t))$ defined on $\mathbb{T}$ which blows up at some finite time, just defining $(\rho(x,t),c(x,t)):=(\rho^1(x_1,t),c^1(x_1,t))$ for $x\in \mathbb{T}^d$, we can check that $(\rho(x,t),c(x,t))$ is a solution to \eqref{H-KS} which undergoes the same blow-up.
\end{remark}
The key of the proof of Theorem \ref{thm: blow-d} is 
\eqref{fb-1-but bounded} which is responsible for the finite propagation speed (Lemma \ref{lemma: fps}).
We prove Theorem \ref{thm: blow-d} in Section \ref{pf: fb-d}.

\medskip

Moreover, proceeding in the same manner as the proof of the previous theorem, we can construct initial data near any constant equilibrium state which blows up in finite time. 
\begin{corollary}[Finite-time blow-up near any constant equilibrium state in $\mathbb{R}^d$]\label{cor: blow-d}
Let $\bar{\rho}>0$ be any constant equilibrium state. Then for any $d\ge1$, $\epsilon>0$, and integer $m\ge0$, there exist some finite time $T^*>0$ and initial data $\left (\rho_{0}, c_{0}\right)$ satisfying \eqref{asm-1}, $(\nabla\rho_0,\nabla c_0)\in (H^{\infty} \times H^{\infty})(\mathbb{R}^d)$, and
\begin{equation}\label{asm: coro}
    \norm{\rho_0-\bar{\rho}}_{H^m(\mathbb{R}^d)}+\norm{\nabla\log c_0}_{H^m(\mathbb{R}^d)} \le \epsilon
\end{equation}
such that the unique solution $(\rho, c)$ to \eqref{H-KS} satisfies
\begin{equation*}
    \lim_{t\to T^{*}} \left({\left\Vert \rho(\cdot,t)  \right\Vert}_{W^{1,\infty}(\mathbb{R}^d)}  +{\left\Vert c(\cdot, t)  \right\Vert}_{W^{2,\infty}(\mathbb{R}^d)}\right) = \infty.
\end{equation*}
\end{corollary}
This corollary gives us an interesting answer to our last question, \textbf{Q3}.
Comparing this corollary with the aforementioned global regularity results near constant equilibrium states for \eqref{P-KS} and \eqref{dH-KS}, we conclude that some kinds of damping or diffusion of $\rho$ are essential to extend the local classical solutions to global ones.
The proof can be found in Section \ref{pf: fb-d}.

\medskip

\subsection*{Notation}
We employ the letter $C=C(a,b,\cdots)$ to denote any constant depending on $a,b,\cdots$, which may change from line to line in a given computation. We sometimes use $A\approx B$ and $A\lesssim B$, which mean $A=CB$ and $A\le CB$, respectively, for some constant $C$.

\section{Local well-posedness}\label{pf: lwp}
In this section, we show Theorem \ref{thm: lwp}. Under the transformation \eqref{C-H trans} and the scalings
\begin{equation}\label{eq: scaling}
    t\rightarrow \sqrt{\chi \mu}t,\quad q \rightarrow \sqrt{\frac{\chi}{\mu}}q,
\end{equation}
\eqref{H-KS} becomes
\begin{equation}\label{H-KS'}
    \left\{
    \begin{aligned}
        &\partial_t \rho = \nabla \cdot (\rho q), \\
        &\partial_t q = \nabla \rho.
    \end{aligned}
    \right.
\end{equation}
We shall divide the proof into two steps, which correspond to a priori estimates and the existence and uniqueness of a solution.

\subsection{A priori estimates}\label{pf: priori}
This subsection is devoted to a priori estimate for a solution $(\rho,c)$ of \eqref{H-KS}, which is assumed to be sufficiently smooth, so that the following computation can be justified.
The equation of $\rho$ in \eqref{H-KS'} gives
\begin{equation}\label{est: rho l infty}
    \frac{d}{dt}\normif{\rho} \lesssim \normif{\nabla \rho}\normif{q} + \normif{\rho}\normif{\nabla q} 
    \lesssim \left(\normif{\rho} + \norm{\nabla\rho}_{H^{m-1}} \right)\norm{q}_{H^m},
\end{equation}
where in the last inequality, we used the Sobolev embedding, $H^{m-1}(\mathbb{R}^d) \hookrightarrow L^{\infty}(\mathbb{R}^d)$.
It is clear that the equation of $c$ in \eqref{H-KS} implies
\begin{equation}\label{est: c l infty}
    \frac{d}{dt}\normif{c} \lesssim \normif{\rho}\normif{c}.
\end{equation}
From the equation of $\rho$ in \eqref{H-KS'}, we can derive
\begin{equation*}
    \partial_{t}\left(\rho^{-1}\right) - q\cdot \nabla \left(\rho^{-1}\right)=-\left(\nabla \cdot q\right)\left(\rho^{-1}\right).
\end{equation*}
Evaluating along the characteristics defined by $-q$, we obtain
\begin{equation}\label{est: rho i l infty}
    \frac{d}{dt}\normif{\rho^{-1}}\lesssim \normif{\nabla q}\normif{\rho^{-1}}\lesssim \norm{q}_{H^m}\normif{\rho^{-1}}.
\end{equation}
On the other hand, since the equation of $c$ in \eqref{H-KS} gives
\begin{equation*}
    \partial_{t}\left(c^{-1}\right)=\mu\rho c^{-1},
\end{equation*}
we have 
\begin{equation}\label{est: c i l infty}
    \frac{d}{dt}\normif{c^{-1}}\lesssim \normif{\rho}\normif{c^{-1}}.
\end{equation}
From \eqref{H-KS'}, we also obtain
\begin{equation}\label{est: q L2}
    \frac{1}{2}\frac{d}{dt}\normb{q}^2 \lesssim \left(\normb{\nabla \rho}^2 + \normb{q}^2 \right).
\end{equation}
Applying $\nabla^m$ to the equation of $\rho$ in \eqref{H-KS'} and taking $L^2$ inner product with $\nabla^m \rho$, we compute
\begin{equation*}
    \begin{split}
        \frac{1}{2}\frac{d}{dt}\normb{\nabla^m \rho}^2=
        \RN{1}_1+ \RN{1}_2 + \int \left( \rho \nabla^{m+1} \cdot q \right) \cdot \nabla^{m} \rho 
    \end{split}
\end{equation*}
with
\begin{equation*}
    \RN{1}_1 = \int \left(\nabla^{m+1} \rho \cdot q \right) \cdot \nabla^{m} \rho \quad \text{and} \quad 
    \RN{1}_2 \approx \int \left(\sum_{1\le |\alpha|\le m}\sum_{|\alpha|+|\beta|=m+1}\nabla^{\alpha} \rho \cdot \nabla^{\beta} q \right) \cdot \nabla^{m} \rho.
\end{equation*}
For $\RN{1}_1$, the integration by parts gives
\begin{equation*}
    \RN{1}_1 = -\int \left(\nabla^{m} \rho \nabla \cdot q \right) \cdot \nabla^{m} \rho - \RN{1}_1,
\end{equation*}
so that we obtain
\begin{equation*}
    \RN{1}_1 = -\frac{1}{2}\int \left(\nabla^{m} \rho \nabla \cdot q \right) \cdot \nabla^{m}\rho  \lesssim \normb{\nabla^m \rho}^2\normif{\nabla q} \lesssim \norm{\nabla \rho}_{H^{m-1}}^2\norm{q}_{H^m}.
\end{equation*}
Using $\norm{fg}_{H^s} \lesssim \normif{f}\norm{g}_{H^s} + \norm{f}_{H^s}\normif{g}$ $(s>0)$,
we estimate
\begin{equation*}
    \RN{1}_2
    \lesssim \left(\normif{\nabla \rho}\norm{\nabla  q}_{H^{m-1}} + \norm{\nabla \rho}_{H^{m-1}}\normif{\nabla  q}\right)\normb{\nabla^{m} \rho}\lesssim \norm{\nabla \rho}_{H^{m-1}}^2\norm{q}_{H^{m}}.
\end{equation*}
Thus, we have
\begin{equation}\label{est:rho m}
    \frac{1}{2}\frac{d}{dt}\normb{\nabla^m \rho}^2-\int \left( \rho \nabla^{m+1} \cdot q \right) \cdot \nabla^{m} \rho
        \lesssim \norm{\nabla \rho}_{H^{m-1}}^2\norm{q}_{H^{m}}.
\end{equation}
On the other hand, using $\rho> 0$ and \eqref{H-KS'}, we compute
\begin{equation*}
    \begin{split}
        \frac{1}{2}\frac{d}{dt}\normb{\sqrt{\rho}\nabla^m q}^2
        = \RN{2}_1+ \RN{2}_2
    \end{split}
\end{equation*}
with
\begin{equation*}
    \RN{2}_1=\frac{1}{2}\int \nabla \cdot(\rho q)  \left(\nabla^m q\right)^2 \quad \text{and} \quad
    \RN{2}_2=\int \rho \nabla^m q \cdot \nabla^{m+1}\rho.
\end{equation*}
For $\RN{2}_1$, we estimate
\begin{equation*}
    \RN{2}_1\lesssim \normif{\nabla(\rho q)}\normb{\nabla^m q}^2  \lesssim \left(\normif{\nabla \rho}\normif{q} + \normif{\rho}\normif{\nabla q}\right) \norm{q}_{H^m}^2  \lesssim \left(\normif{\rho} + \norm{\nabla\rho}_{H^{m-1}} \right)\norm{q}_{H^m}^3.
\end{equation*}
Using the integration by parts, we compute
\begin{equation*}
    \RN{2}_2=-\int \left(\nabla \rho \cdot\nabla^m q \right) \cdot \nabla^{m}\rho
    -\int \left( \rho \nabla^{m+1} \cdot q \right) \cdot \nabla^{m} \rho,
\end{equation*}
so that we have
\begin{equation*}
    \RN{2}_2 +\int \left( \rho \nabla^{m+1} \cdot q \right) \cdot \nabla^{m} \rho\lesssim\norm{\nabla \rho}_{H^{m-1}}^2\norm{q}_{H^{m}}.
\end{equation*}
Thus, we obtain
\begin{equation}\label{est:q m}
    \begin{split}
        \frac{1}{2}\frac{d}{dt}\normb{\sqrt{\rho}\nabla^m q}^2+\int \left( \rho \nabla^{m+1} \cdot q \right) \cdot \nabla^{m} \rho
        \lesssim \left(\normif{\rho} + \norm{\nabla\rho}_{H^{m-1}} \right)\norm{q}_{H^m}^3 +\norm{\nabla \rho}_{H^{m-1}}^2\norm{q}_{H^{m}}.
    \end{split}
\end{equation}
Adding \eqref{est:rho m} to \eqref{est:q m}, we have
\begin{equation*}
    \frac{d}{dt}\left(\normb{\nabla^m \rho}^2 + \normb{\sqrt{\rho}\nabla^m q}^2\right)
    \lesssim \left(\normif{\rho} + \norm{\nabla\rho}_{H^{m-1}} \right)\norm{q}_{H^m}^3 +\norm{\nabla \rho}_{H^{m-1}}^2\norm{q}_{H^{m}}.
\end{equation*}
With the same argument, we can obtain
\begin{equation}\label{est: m}
\frac{d}{dt}\sum_{k=1}^{m}\left(\normb{\nabla^k \rho}^2 + \normb{\sqrt{\rho}\nabla^k q}^2 \right)
\lesssim  \left(\normif{\rho} + \norm{\nabla\rho}_{H^{m-1}} \right)\norm{q}_{H^m}^3 +\norm{\nabla \rho}_{H^{m-1}}^2\norm{q}_{H^{m}}.
\end{equation}
We now define
\begin{equation*}
    X_m:= 1+ \normif{\rho} + \normif{c} + \normif{\rho^{-1}} + \normif{c^{-1}} + \normb{q}
    +\sum_{k=1}^{m}\left(\normb{\nabla^k \rho}^2 + \normb{\sqrt{\rho}\nabla^k q}^2 \right).
\end{equation*}
Then, since
\begin{equation*}
    \norm{q}_{H^m}\approx \normb{q}
    +\normb{(\sqrt{\rho})^{-1}\sqrt{\rho}\nabla^m q} \lesssim \normb{q}
    +\sqrt{\normif{\rho^{-1}}}\normb{\sqrt{\rho}\nabla^m q} \lesssim X_m,
\end{equation*}
\eqref{est: rho l infty}-\eqref{est: q L2}, and \eqref{est: m} imply that
\begin{equation*}\label{est: X}
    \frac{d}{dt}X_m \lesssim X_m^4.
\end{equation*}
This shows that formally, there exists a time interval $[0,T]$ in which $\normif{\rho}$, $\normif{c}$, 
$\normif{\rho^{-1}}$, $\normif{c^{-1}}$, $\norm{\nabla \rho}_{H^{m-1}}$, and $\norm{q}_{H^{m}}$ remain finite. Furthermore, we can check that $\norm{\nabla c}_{H^{m}}$ also remains finite on the same time interval, noticing that the definition of $q$ ensures the existence of a constant $C=C\left(\normif{c}, \normif{c^{-1}}, \norm{q}_{H^m}\right)>0$ such that $\norm{\nabla c}_{H^{m}} \le C$.

In order to show that $(\nabla \rho_0,\nabla c_0) \in H^{\infty} \times H^{\infty}$ implies $(\nabla \rho (t,\cdot),\nabla c(t,\cdot)) \in H^{\infty} \times H^{\infty}$ on the same time interval, we use the induction argument. Assume that $\normif{\rho}$, $\normif{c}$, 
$\normif{\rho^{-1}}$, $\normif{c^{-1}}$, $\norm{\nabla \rho}_{H^{m-1}}$, $\norm{q}_{H^{m}}$, and $\norm{\nabla c}_{H^{m}}$ are finite on $[0,T]$.
Then, using the cancellation as we did above, we can obtain
\begin{equation*}
    \frac{1}{2}\frac{d}{dt}\left(\normb{\nabla^{m+1} \rho}^2+ \normb{\sqrt{\rho}\nabla^{m+1} q}^2 \right)=
        \RN{3}_1+ \RN{3}_2 + \RN{4}_1 + \RN{4}_2
\end{equation*}
with
\begin{equation*}
    \RN{3}_1 = \int \left(\nabla^{m+2} \rho \cdot q \right) \cdot \nabla^{m+1} \rho, \quad \RN{3}_2 \approx \int \left(\sum_{1\le |\alpha|\le m+1}\sum_{|\alpha|+|\beta|=m+2}\nabla^{\alpha} \rho \cdot \nabla^{\beta} q \right) \cdot \nabla^{m+1} \rho,
\end{equation*}
\begin{equation*}
    \RN{4}_1=\frac{1}{2}\int \left(\nabla \cdot(\rho q) \right) \left(\nabla^{m+1} q\right)^2, \quad \text{and} \quad \RN{4}_2=\int \left(\nabla \rho \cdot\nabla^{m+1} q \right) \cdot \nabla^{m+1}\rho.
\end{equation*}
For $\RN{3}_1$, using the integration by parts and the assumption that $\normif{\nabla q} \lesssim \norm{q}_{H^m} \le C$, we have
\begin{equation*}
    \RN{3}_1 = -\frac{1}{2}\int \left(\nabla^{m+1} \rho \nabla \cdot q \right) \cdot \nabla^{m+1}\rho  \le C \normb{\nabla^{m+1} \rho}^2.
\end{equation*}
We decompose $\RN{3}_2$ into 
\begin{equation*}
    \RN{3}_2 = \RN{3}_{21} + \RN{3}_{22}
\end{equation*}
with
\begin{equation*}
    \RN{3}_{21}\approx \int \left(\sum_{|\alpha|=1}\sum_{|\beta|=m+1}\nabla^{\alpha} \rho \cdot \nabla^{\beta} q \right) \cdot \nabla^{m+1} \rho
    +\int \left(\sum_{|\alpha|=m+1}\sum_{|\beta|=1}\nabla^{\alpha} \rho \cdot \nabla^{\beta} q \right) \cdot \nabla^{m+1} \rho,
\end{equation*}
and
\begin{equation*}
    \RN{3}_{22}\approx \int \left(\sum_{2\le |\alpha|\le m}\sum_{|\alpha|+|\beta|=m+2}\nabla^{\alpha} \rho \cdot \nabla^{\beta} q \right) \cdot \nabla^{m+1} \rho.
\end{equation*}
Noticing the assumption that $\normif{\nabla \rho} \lesssim \norm{\rho}_{H^m} \le C$ and  $\normif{\nabla q} \lesssim \norm{q}_{H^m} \le C$, we estimate
\begin{equation*}
    \RN{3}_{21} \lesssim \normb{\nabla^{m+1} \rho}\normb{\nabla^{m+1} q} \lesssim \normb{\nabla^{m+1} \rho}^2 + \normb{\nabla^{m+1} q}^2.
\end{equation*}
Since $\normb{\nabla^{\alpha}\rho}\le\normb{\nabla^{m+1}\rho}$ for $2\le |\alpha|\le m$, $\norm{q}_{H^m}\le C$ again gives $\RN{3}_{22} \lesssim \normb{\nabla^{m+1} \rho}^2.$
Considering $\normif{\nabla \cdot(\rho q)}\lesssim \norm{\rho}_{H^m} \norm{q}_{H^m} \le C$, we have
$\RN{4}_1\lesssim \normb{\nabla^{m+1}q}^2.$
Using $\normif{\nabla \rho}\lesssim \norm{\rho}_{H^m} \le C$ again, we estimate
\begin{equation*}
    \RN{4}_{2} \lesssim \normb{\nabla^{m+1} \rho}\normb{\nabla^{m+1} q} \lesssim \normb{\nabla^{m+1} \rho}^2 + \normb{\nabla^{m+1} q}^2.
\end{equation*}
Combining all, and noticing 
\begin{equation*}
    \normb{\nabla^{m+1} q}^2 = \normb{(\sqrt{\rho})^{-1}\sqrt{\rho}\nabla^{m+1} q} 
    \le \sqrt{\normif{\rho^{-1}}}\normb{\sqrt{\rho}\nabla^{m+1} q}^2
    \lesssim \normb{\sqrt{\rho}\nabla^{m+1} q}^2,
\end{equation*}
we finally have arrived at
\begin{equation*}
    \frac{d}{dt}\left(\normb{\nabla^{m+1} \rho}^2+ \normb{\sqrt{\rho}\nabla^{m+1} q}^2\right) \lesssim \normb{\nabla^{m+1} \rho}^2+ \normb{\sqrt{\rho}\nabla^{m+1} q}^2. 
\end{equation*}
Hence, the Gr\"onwall's inequality implies that $\normb{\nabla^{m+1} \rho}$ and $\normb{\nabla^{m+1} q}$ remain finite on $[0,T]$. $\Box$
\begin{remark}\label{rmk: reg.-1}
Observing our priori estimate, we can see that 
\begin{equation*}
    \sup_{t\in[0,T)}\norm{\rho(t)}_{W^{1,\infty}(\mathbb{R}^d)}  + \norm{c(t)}_{W^{2,\infty}(\mathbb{R}^d)}
\end{equation*}
controls blow-up. In other words, for $(\rho_0,c_0)$ satisfying \eqref{asm-1} and \eqref{asm-lwp-1}, we have corresponding solution $(\rho,c)$ satisfying \eqref{asm-lwp-2} on $[0,T]$ as long as
\begin{equation}\label{blowup criterion}
    \norm{\rho(t)}_{W^{1,\infty}(\mathbb{R}^d)}  + \norm{c(t)}_{W^{2,\infty}(\mathbb{R}^d)} < \infty \quad \text{on $[0,T]$.}
\end{equation}
We can also ensure $\left(\nabla \rho, \nabla c\right) \in L^{\infty}\left([0,T];(H^{\infty} \times H^{\infty})(\mathbb{R}^d)\right)$ if
$(\rho_0,c_0)$ further satisfies $(\nabla\rho_0,\nabla c_0)\in (H^{\infty} \times H^{\infty})(\mathbb{R}^d)$ and \eqref{blowup criterion} holds.
\end{remark}

\subsection{Existence and Uniqueness}\label{pf: ex-uni}
To begin with, the proof of existence can be done using viscous approximation (refer to \cite{IJ21}):
For fixed $\epsilon>0$, we consider
\begin{equation}\label{P-KS-eps}
\left\{
\begin{aligned}
	&\partial_{t} \rho^{(\epsilon)} = \epsilon\Delta \rho^{(\epsilon)} - \chi\nabla \cdot \left (\rho^{(\epsilon)} \nabla \log c^{(\epsilon)}\right), \\
	&\partial_{t} c^{(\epsilon)} = -\mu c \rho^{(\epsilon)}. \\
\end{aligned}
\right.
\end{equation}
with the same initial data $(\rho_0,c_0)$.
Under the transformation \eqref{C-H trans} and  the scaling \eqref{eq: scaling}, \eqref{P-KS-eps} becomes
\begin{equation}\label{H-KS'-eps}
    \left\{
    \begin{aligned}
        &\partial_t \rho^{(\epsilon)} = \frac{\epsilon}{\sqrt{\chi \mu}}\Delta \rho^{(\epsilon)} + \nabla \cdot (\rho^{(\epsilon)} q^{(\epsilon)}), \\
        &\partial_t q^{(\epsilon)} = \nabla \rho^{(\epsilon)}.
    \end{aligned}
    \right.
\end{equation}
Existence of smooth local in time solutions to \eqref{P-KS-eps} and \eqref{H-KS'-eps} was already established in \cite{LLZ11}. Furthermore, we can show that the solution satisfies our priori estimates of the last subsection uniformly in $\epsilon>0$. See \cite{IJ21} for details.

To prove uniqueness, we assume that there exist two solutions $(\rho_i,c_i)$ $(i=1,2)$ on $[0,T]$ to \eqref{H-KS}  satisfying \eqref{asm-lwp-2} and $(\rho_1(t=0),c_1(t=0))=(\rho_2(t=0),c_2(t=0))$.
We denote
\begin{equation}\label{def: difference}
    q_i=-\frac{\nabla c_i}{c_i}\,(i=1,2),\quad \tilde{\rho}:=\rho_1-\rho_2,\quad \tilde{c}:=c_1-c_2, \quad \text{and}\quad \tilde{q}:=q_1-q_2.
\end{equation}
From the equations of $\tilde{\rho}$:
\begin{equation}\label{eq: tilde rho}
    \partial_t \tilde{\rho}=\nabla\tilde{\rho}\cdot q_1 
    +\tilde{\rho} \nabla \cdot q_1
    +\nabla\rho_2\cdot \tilde{q}
    +\rho_2 \nabla \cdot \tilde{q}, \quad 
\end{equation}
and the equation of $\tilde{q}$:
\begin{equation}\label{eq: tilde q}
    \partial_t \tilde{q}=\nabla \tilde{\rho},
\end{equation}
we have
\begin{equation*}
    \frac{1}{2}\frac{d}{dt}\left(\normb{\tilde\rho}^2 + \normb{\sqrt{\rho_2}\tilde q}^2 \right)
    =\RN{1}+\RN{2}+\RN{3}+\RN{4}.
\end{equation*}
with
\begin{equation*}
    \RN{1}=\int \tilde{\rho} \nabla\tilde{\rho}\cdot q_1+\int \tilde{\rho}^2 \nabla \cdot q_1,
    \quad \RN{2}=\int  \tilde{\rho}\nabla\rho_2\cdot \tilde{q} + \int \tilde{\rho}\rho_2 \nabla \cdot \tilde{q}, \quad
    \RN{3}=\frac{1}{2}\int  |\tilde{q}|^2\nabla \cdot (\rho_2 q_2), \quad \text{and} \quad \RN{4}=\int \rho_2 \tilde{q} \cdot \nabla \tilde{\rho}.
\end{equation*}
Using the integration by parts on the first term of $\RN{1}$, we estimate
\begin{equation*}
    \RN{1}=\frac{1}{2}\int \nabla\left(\tilde{\rho}^2\right) \cdot q_1
    +\int \tilde{\rho}^2 \nabla \cdot q_1
    = -\frac{1}{2} \int \tilde{\rho}^2 \nabla \cdot q_1 +\int \tilde{\rho}^2 \nabla \cdot q_1
    = \frac{1}{2} \int \tilde{\rho}^2 \nabla \cdot q_1
    \le \norm{q_1}_{W^{1,\infty}}\normb{\tilde{\rho}}^2.
\end{equation*}
The integration by parts also gives $\RN{2}+\RN{4}=0$.
Moreover, we have
\begin{equation*}
    \RN{3} \le \normif{\rho_2^{-1}}\normb{\sqrt{\rho_2}\tilde q}^2\norm{\rho_2}_{W^{1,\infty}}\norm{q_2}_{W^{1,\infty}}.
\end{equation*}
Note that $\norm{q_1}_{W^{1,\infty}}$, $\norm{\rho_2}_{W^{1,\infty}}$, $\norm{q_2}_{W^{1,\infty}}$, and
$\normif{\rho_2^{-1}}$ are all bounded, and therefore,
\begin{equation*}
    \frac{1}{2}\frac{d}{dt}\left(\normb{\tilde\rho}^2 + \normb{\sqrt{\rho_2}\tilde q}^2 \right)
   \lesssim \normb{\tilde\rho}^2 + \normb{\sqrt{\rho_2}\tilde q}^2.
\end{equation*}
Since $\left(\normb{\tilde\rho}^2 + \normb{\sqrt{\rho_2}\tilde q}^2\right)(0)=0$, we conclude that
$\left(\normb{\tilde\rho}^2 + \normb{\sqrt{\rho_2}\tilde q}^2\right)(t)=0$ for $t\in[0,T]$.
To show $\tilde c(t)=0$ for $t\in[0,T]$, we consider the equation of $\Tilde{c}$:
\begin{equation*}
    \partial_t \Tilde{c} = -\mu\left(\Tilde{c}\rho_1 + c_2\Tilde{\rho}\right)=-\mu\Tilde{c}\rho_1.
\end{equation*}
The second equality follows from $\Tilde{\rho}=0$ on $[0,T]$.
Thus using the Gr\"onwall's inequality, $\Tilde{c(t)}=0$ on $[0,T]$. $\Box$
\begin{remark}\label{rmk: reg.-2}
As we can see from the above proof, we are able to prove that the solution is unique on $[0,T]$
as long as the solution satisfies
\begin{equation*}
    \norm{\rho(t)}_{W^{1,\infty}(\mathbb{R}^d)}  +\norm{c(t)}_{W^{2,\infty}(\mathbb{R}^d)} < \infty \quad \text{on $[0,T]$.}
\end{equation*}
\end{remark}

\section{Sufficient conditions of data for finite-time blow-up in $\mathbb{R}$}\label{pf: fb-1}
In this section, we prove Theorem \ref{thm: blow-1} with the aid of Riemann invariants. 
Henceforth, we always assume that initial data $(\rho_{0},c_{0})$ satisfies \eqref{asm-1}, \eqref{asm-2}-\eqref{asm-4}. 

We firstly show the existence of finite time $T^*>0$ such that
\begin{equation}\label{fb-1-1}
    \lim_{t\to T^{*}} \left({\left\Vert \rho(t)  \right\Vert}_{W^{1,\infty}(\mathbb{R})}  +{\left\Vert c(t)  \right\Vert}_{W^{2,\infty}(\mathbb{R})}\right) = \infty.
\end{equation}
Then we prove that the solution actually satisfies \eqref{fb-1-but bounded} for $t\in[0,T^*]$.

Now we begin the proof.
Suppose, on the contrary, that for any finite time $T>0$, the unique solution corresponding to these initial data satisfies
\begin{equation}\label{f-r1-asm-1}
    {\left\Vert \rho(t)  \right\Vert}_{W^{1,\infty}(\mathbb{R})}  +{\left\Vert c(t)  \right\Vert}_{W^{2,\infty}(\mathbb{R})} < \infty \quad \text{on $[0,T]$.}
\end{equation}
Recall that \eqref{f-r1-asm-1} is sufficient to guarantee that $(\rho,c)$ is smooth and unique, and has the positive lower bounds on $[0,T]$ by Remark \ref{rmk: reg.-1} and Remark \ref{rmk: reg.-2}.
Abusing the notation only in this section, we denote $q:=\frac{\partial_x c}{c}$ (not $q=-\frac{\partial_x c}{c}$), and
change \eqref{H-KS'} into the matrix form:
\begin{equation*}
    \partial_{t}\begin{pmatrix}
        \rho\\
        q
    \end{pmatrix}
    + \begin{pmatrix}
    q & \rho \\
    1 & 0
\end{pmatrix}
\partial_{x}
\begin{pmatrix}
    \rho \\
    q
\end{pmatrix}
=\begin{pmatrix}
    0\\
    0
\end{pmatrix}.
\end{equation*}
Note that
$\begin{pmatrix}
    q & \rho \\
    1 & 0
\end{pmatrix}$
has eigenvalues: $\lambda_{1}(\rho,q)=\frac{q-\sqrt{q^2+4\rho}}{2}$, $\lambda_{2}(\rho,q)=\frac{q+\sqrt{q^2+4\rho}}{2}$, and corresponding eigenvectors: $r_{1}(\rho,q)=\left(\frac{q-\sqrt{q^2+4\rho}}{2},1\right)$, $r_{2}(\rho,q)=\left(\frac{q+\sqrt{q^2+4\rho}}{2},1\right)$.
We can regard $\lambda_{i}$ and $r_{i}$ $(i=1,2)$ as functions on $ \Omega:=\left\{(z_{1},z_{2})\subset \mathbb{R}^2 : z_{1}>0 \right\}$. \\

We now construct a new coordinate system $w$ on $\Omega$, called Riemann invariants.
\begin{lemma}\label{lem:RI}
There exists a global diffeomorphism $w : \Omega \rightarrow w(\Omega)\subset \mathbb{R}^2$ such that on $\Omega$,
\begin{subequations}
    \begin{align}
    &\bullet \; \nabla w_{i} \cdot r_{i} = 0 \quad  (i=1,2), \label{RE: p1}\\
    &\bullet \; \frac{\partial w_{1}}{\partial z_{1}}>0, \quad \frac{\partial w_{2}}{\partial z_{1}}<0.
    \label{RE: p2}
\end{align}
\end{subequations}
\end{lemma}
\begin{proof}
To solve the following PDE on $\Omega$:
\begin{equation}\label{f1}
    \left\{
\begin{aligned}
	&\partial_{z_2}f_1\left (z_1, z_2\right)+\frac{z_2-\sqrt{{z_2}^2+4z_1}}{2} \partial_{z_1}f_1\left (z_1, z_2\right)= \frac{1}{\sqrt{z_2^2+4z_1}}f_1\left (z_1, z_2\right), \\
	&f_1(z_1,0)=e^{z_1},
\end{aligned}
\right.
\end{equation}
we consider the characteristic $\phi_1(z_1,z_2)$ defined by
\begin{equation}\label{ode for phi}
    \left\{
    \begin{aligned}
        &\partial_{z_2}\phi_1\left (z_1, z_2\right)=\frac{z_2-\sqrt{{z_2}^2+4\phi_1(z_1,z_2)}}{2}, \\
        & \phi_1(z_1,0)=z_1.
    \end{aligned}
    \right.
\end{equation}
(We observe that given $z_1>0$, $\phi_1(z_1,z_2)$ is defined for all $z_2\in \mathbb{R}$. 
Denoting $F(z_1,z_2):=\frac{z_2-\sqrt{{z_2}^2+4z_1}}{2}$, we can check that $F$ is continuous in $\Omega$ and $F(\cdot,z_2)$ is globally Lipschitz continuous in $\mathbb{R}_{>0}$ for each $z_2\neq 0$.
Noticing $\partial_{z_2}\phi_1\left (z_1, z_2\right)=F(\phi_1(z_1,z_2),z_2)$, we can first use Peano theorem to obtain an interval of existence near $z_2=0$, say $(-a,a)$, and then apply global version of the Picard-Lindel\"of theorem to $\mathbb{R}\backslash [-a/2,a/2]$ to see that $\phi_1$ is defined for all $z_2 \in \mathbb{R}$.) Evaluating along $\phi_1$,
we can check that $f_1$ is smooth and $f_1>0$ on $\Omega$.
Furthermore, since \eqref{f1} implies that
\begin{equation*}
     \partial_{z_2}f_1\left (z_1, z_2\right)=\partial_{z_1} \left(\frac{-z_2+\sqrt{{z_2}^2+4z_1}}{2} f_1\left (z_1, z_2\right) \right),
\end{equation*}
the Poincar\'e lemma ensures the existence of a smooth function $w_1 : \Omega \rightarrow \mathbb{R}$ satisfying
\begin{equation}\label{w_1}
    \frac{\partial w_{1}}{\partial z_{1}}=f_1\left (z_1, z_2\right), \quad \frac{\partial w_{1}}{\partial z_{2}}=\frac{-z_2+\sqrt{{z_2}^2+4z_1}}{2} f_1\left (z_1, z_2\right).
\end{equation}
This gives us that $\nabla w_{1} \cdot r_{1} = 0$.

\noindent Similarly, with the aid of the characteristic $\phi_2(z_1,z_2)$ defined by
\begin{equation*}
    \left\{
    \begin{aligned}
        &\partial_{z_2}\phi_2\left (z_1, z_2\right)=\frac{z_2+\sqrt{{z_2}^2+4\phi_2(z_1,z_2)}}{2},\\
        &\phi_2(z_1,0)=z_1,
    \end{aligned}
    \right.
\end{equation*}
we can again see that $\phi_2$ is defined for all $z_2  \in \mathbb{R}$ and 
obtain a negative smooth function $f_2$ defined on $\Omega$ solving
\begin{equation}\label{f2}
    \left\{
    \begin{aligned}
        &\partial_{z_2}f_2\left (z_1, z_2\right)+\frac{z_2+\sqrt{{z_2}^2+4z_1}}{2} \partial_{z_1}f_1\left (z_1, z_2\right)= -\frac{1}{\sqrt{z_2^2+4z_1}}f_2\left (z_1, z_2\right),\\
        &f_1(z_1,0)=-e^{z_1}.
    \end{aligned}
    \right.
\end{equation}
Since \eqref{f2} implies that
\begin{equation*}
     \partial_{z_2}f_2\left (z_1, z_2\right)=\partial_{z_1} \left(-\frac{z_2+\sqrt{{z_2}^2+4z_1}}{2} f_2\left (z_1, z_2\right) \right),
\end{equation*}
the Poincar\'e lemma again implies that there exists a smooth function $w_2 : \Omega \rightarrow \mathbb{R}$ satisfying
\begin{equation*}
    \frac{\partial w_{2}}{\partial z_{1}}=f_2\left (z_1, z_2\right), \quad \frac{\partial w_{2}}{\partial z_{2}}=-\frac{z_2+\sqrt{{z_2}^2+4z_1}}{2} f_2\left (z_1, z_2\right),
\end{equation*}
which yields $\nabla w_{2} \cdot r_{2} = 0$.

Now we show that $w=(w_1,w_2): \Omega \rightarrow w(\Omega)\subset \mathbb{R}^2$ is a global diffeomorphism. Since $f_1 >0$, $f_2 <0$, and 
\begin{equation}\label{nabla w}
    \nabla w =\left(
    \begin{aligned}
        &\frac{\partial w_{1}}{\partial z_{1}} \;\, \frac{\partial w_{1}}{\partial z_{2}}\\
        &\frac{\partial w_{2}}{\partial z_{1}} \;\, \frac{\partial w_{2}}{\partial z_{2}}
    \end{aligned}
    \right)
    =\left(
    \begin{aligned}
        &f_{1}(z_1,z_2) \quad f_{1}(z_1,z_2)\frac{-z_2+\sqrt{{z_2}^2+4z_1}}{2}\\
        &f_{2}(z_1,z_2) \quad -f_{2}(z_1,z_2)\frac{z_2+\sqrt{{z_2}^2+4z_1}}{2}
    \end{aligned}
    \right),
\end{equation}
we have
\begin{equation*}
    \det \left(\nabla w \right)= -f_1(z_1,z_2)f_2(z_1,z_2)\sqrt{{z_2}^2+4z_1} >0 \quad \text{on $\Omega$},
\end{equation*}
so that the inverse function theorem implies that $w$ is a local diffeomorphism. Thus, it suffices to prove that $w$ is (globally) one-to-one. Suppose, on the contrary, that $w(\bar{z_1},\bar{z_2})=w(\tilde{z_1},\tilde{z_2})$ for some $(\bar{z_1},\bar{z_2})\neq (\tilde{z_1},\tilde{z_2})$. Then by the mean value theorem, there exist two points $(z_1^*,z_2^*)$ and $(z_1^{**},z_2^{**})$ lying on the line segment with endpoints $(\bar{z_1},\bar{z_2})$ and $ (\tilde{z_1},\tilde{z_2})$ such that
\begin{align}
    0= w_1(\bar{z_1},\bar{z_2})-w_1(\tilde{z_1},\tilde{z_2}) &= \nabla w_1|_{(z_1^*,z_2^*)} \cdot (\bar{z_1}-\tilde{z_1},\bar{z_2}-\tilde{z_2}) \nonumber \\
    &=f_{1}(z_1^*,z_2^*)(\bar{z_1}-\tilde{z_1})+ f_{1}(z_1^*,z_2^*)\frac{-z_2^*+\sqrt{{z_2^*}^2+4z_1^*}}{2}(\bar{z_2}-\tilde{z_2}), \label{eq: w_1-lma}
\end{align}
and
\begin{align}
    0= w_2(\bar{z_1},\bar{z_2})-w_2(\tilde{z_1},\tilde{z_2}) &= \nabla w_2|_{(z_1^{**},z_2^{**})} \cdot (\bar{z_1}-\tilde{z_1},\bar{z_2}-\tilde{z_2}) \nonumber \\
    &=f_{2}(z_1^{**},z_2^{**})(\bar{z_1}-\tilde{z_1})- f_{2}(z_1^{**},z_2^{**})\frac{z_2^{**}+\sqrt{{z_2^{**}}^2+4z_1^{**}}}{2}(\bar{z_2}-\tilde{z_2}). \label{eq: w_2-lma}
\end{align}
If $\bar{z_1}<\tilde{z_1}$ and $\bar{z_2}\le \tilde{z_2}$, then \eqref{eq: w_1-lma} is less than $0$, which is a contradiction. If $\bar{z_1}<\tilde{z_1}$ and $\bar{z_2} > \tilde{z_2}$, then \eqref{eq: w_2-lma} is less than $0$, which is also a contradiction. We can similarly derive contradictions in the other cases, which implies that $(\bar{z_1},\bar{z_2})= (\tilde{z_1},\tilde{z_2})$.
\end{proof}

\begin{remark}
Lemma \ref{lem:RI} enables us to write that for $(z_1,z_2)\in \Omega$, 
\begin{equation*}
    \lambda_i(z_1,z_2)=\left(\lambda_i \circ w^{-1}\right)\left(w_1(z_1,z_2),w_2(z_1,z_2)\right), \quad (i=1,2),
\end{equation*}
and ensures that each $\lambda_{i}$ has the same regularity with $\lambda_{i} \circ w^{-1}$. Henceforth, we identify $\lambda_{i}$ defined on $\Omega$ with $\lambda_{i} \circ w^{-1}$ defined on $w(\Omega)$, and regard $\lambda_{i}$ as a function of $(w_1,w_2)\in w(\Omega)$ as well as $(z_1,z_2)\in \Omega$.
\end{remark} 
Defining a pair of functions $(P,Q):\mathbb{R} \times [0,\infty) \rightarrow w(\Omega)$ by
\begin{equation*}\label{eq: PQ}
    P(x,t)=w_1(\rho(x,t),q(x,t)), \quad Q(x,t)=w_2(\rho(x,t),q(x,t)),
\end{equation*}
we can check that $(P,Q)$ solves
\begin{equation}\label{pde: PQ}
    \left\{
    \begin{aligned}
        &\partial_{t} P + \lambda_{2}(\rho,q) \partial_{x} P =0 ,\\
        &\partial_{t} Q + \lambda_{1}(\rho,q) \partial_{x} Q =0,
    \end{aligned}
    \right.
\end{equation}
with the aid of \eqref{RE: p1}. Since this is Theorem 1 in Section 11.3 of \cite{Eva}, we omit the details.

For the proof of Theorem \ref{thm: blow-1}, we need two key ingredients.
\begin{lemma}\label{lma: dlt_0}
There exists a universal constant $\delta_0 >0$ such that 
\begin{equation*}
    \frac{\partial\lambda_{2}}{\partial w_1}(P(x,t),Q(x,t))\ge \delta_0 \quad \text{in} \;\, \mathbb{R} \times [0,\infty).
\end{equation*}
\end{lemma}
\begin{proof}
To begin with, we claim that 
\begin{equation}\label{lammda/w}
  \frac{\partial\lambda_{2}}{\partial w_1}>0 \quad \text{on} \;\,  w(\Omega).
\end{equation}
Indeed, since \eqref{nabla w} provides us with
\begin{equation*}
    \left(
    \begin{aligned}
        &\frac{\partial z_{1}}{\partial w_{1}} \;\, \frac{\partial z_{1}}{\partial w_{2}} \\
        &\frac{\partial z_{2}}{\partial w_{1}} \;\, \frac{\partial z_{2}}{\partial w_{2}}
    \end{aligned}
    \right)
    =(\nabla w)^{-1}
    =\left(
    \begin{aligned}
        &\frac{z_2+\sqrt{z_{2}^{2}+4z_1}}{2f_{1}(z_1,z_2)\sqrt{z_{2}^{2}+4z_1}}
        \quad \frac{-z_2+\sqrt{z_{2}^{2}+4z_1}}{2f_{2}(z_1,z_2)\sqrt{z_{2}^{2}+4z_1}} \\
        &\frac{1}{f_{1}(z_1,z_2)\sqrt{z_{2}^{2}+4z_1}}
        \quad \frac{-1}{f_{2}(z_1,z_2)\sqrt{z_{2}^{2}+4z_1}}
    \end{aligned}
    \right),
\end{equation*}
we have
\begin{equation*}
    \frac{\partial\lambda_{2}}{\partial w_1}=\frac{\partial \lambda_{2}}{\partial z_{1}}\frac{\partial z_{1}}{\partial w_{1}}
    +\frac{\partial \lambda_{2}}{\partial z_{2}}\frac{\partial z_{2}}{\partial w_{1}}=\frac{z_{2}+\sqrt{z_{2}^{2}+4z_1}}{f_{1}(z_1,z_2)(z_{2}^{2}+4z_1)}>0.
\end{equation*}
Evaluating along the characteristics defined by each $\lambda_{i}(\rho,q)$ $(i=1,2)$, we obtain from \eqref{pde: PQ} that 
\begin{equation*}
P\left(\mathbb{R} \times [0,\infty)\right)=P_{0}\left(\mathbb{R}\right)=w_1\left(\rho_{0}\left(\mathbb{R}\right),q_{0}\left(\mathbb{R}\right)\right)=w_1\left(\rho_{0}\left(\mathbb{R}\right),\frac{\partial_{x}c_0}{c_0}\left(\mathbb{R}\right)\right),
\end{equation*}
and
\begin{equation*}
Q\left(\mathbb{R} \times [0,\infty)\right)=Q_{0}\left(\mathbb{R}\right)=w_2\left(\rho_{0}\left(\mathbb{R}\right),q_{0}\left(\mathbb{R}\right)\right)=w_2\left(\rho_{0}\left(\mathbb{R}\right),\frac{\partial_{x}c_0}{c_0}\left(\mathbb{R}\right)\right),
\end{equation*}
where $P_0(x):=P(x,0)$ and $Q_0(x):=Q(x,0)$. Since each $w_i$ is smooth and both $\rho_{0}\left(\mathbb{R}\right)$ and $\frac{\partial_{x}c_0}{c_0}\left(\mathbb{R}\right)$ are compact by the assumptions \eqref{asm-1}, \eqref{asm-2}, and \eqref{asm-3}, $P\left(\mathbb{R} \times [0,\infty)\right)$ and $Q\left(\mathbb{R} \times [0,\infty)\right)$ are also compact. Thus, \eqref{lammda/w} and the smoothness of $\frac{\partial\lambda_{2}}{\partial w_1}$ ensure the existence of the desired $\delta_0>0$.
\end{proof}

\begin{lemma}\label{lma: x_0}
Let $x_0$ be a point satisfying \eqref{asm-4}. Then
\begin{equation*}
    \partial_{x}P_0(x_0)<0.
\end{equation*}
\end{lemma}
\begin{proof}
By \eqref{w_1}, \eqref{RE: p2}, \eqref{asm-1}, and \eqref{asm-4}, we compute at $x_0$
\begin{align*}
    \partial_{x}P_{0}&=\frac{\partial w_1}{\partial z_1} \, \partial_x \rho_0 + \frac{\partial w_1}{\partial z_2} \, \partial_x q_0 \\
    &= f_1(\rho_0,q_0) \,\partial_x\rho_0 + f_1(\rho_0,q_0) \, \frac{-q_0+\sqrt{q_0^2+4\rho_0}}{2}\,\frac{c_0\partial_{xx}c_0-(\partial_{x}c_0)^2}{c_0^2} <0.
\end{align*}
\end{proof}

We are now in the position to prove Theorem \ref{thm: blow-1}. We use a similar argument to the proof of Theorem 2 in section 11.3 of \cite{Eva}.
Let $x_{0}$ be the point from \eqref{asm-4}, and let $x_{\lambda_2}(t)$ be the characteristic curve defined by $\lambda_2(\rho,q)$ with initial value $x_{\lambda_2}(0)=x_0.$ Note that $P(x_{\lambda_2}(t),t)=P_{0}(x_0)$ for $t\ge0$.
For the simplicity, we write $\tilde{P}:=\partial_x P$, $\tilde{Q}:=\partial_x Q$, and $\tilde{P}_{\lambda_2}(t):=\partial_x P (x_{\lambda_2}(t),t)$. Then \eqref{pde: PQ} implies that
\begin{equation*}
    \partial_{t}\tilde{P} + \lambda_2\partial_{x}\tilde{P} + \frac{\partial\lambda_{2}}{\partial w_1} \tilde{P}^2 + \frac{\partial\lambda_{2}}{\partial w_2} \tilde{P} \tilde{Q}=0,
\end{equation*}
and
\begin{equation*}
    \partial_{t}Q + \lambda_{2} \partial_{x} Q = (\lambda_2 - \lambda_1)\tilde{Q}.
\end{equation*}
Combining these, we obtain
\begin{equation*}\label{eq: tilde PQ}
    \partial_{t}\tilde{P} + \lambda_2\partial_{x}\tilde{P} + \frac{\partial\lambda_{2}}{\partial w_1} \tilde{P}^2 + \left[\frac{1}{\lambda_2-\lambda_1}\frac{\partial\lambda_{2}}{\partial w_2}\left(\partial_{t}Q + \lambda_{2} \partial_{x} Q\right)\right] \tilde{P} =0.
\end{equation*}
Defining a function $\Phi(t)$ by 
\begin{equation*}\label{def: Phi}
    \Phi(t):=\exp{\left(\int_{0}^{t}\frac{1}{\lambda_2-\lambda_1}\frac{\partial\lambda_{2}}{\partial w_2}\left(\partial_{t}Q + \lambda_{2} \partial_{x} Q\right)\left(x_{\lambda_{2}}(s),s\right)\; ds\right)},
\end{equation*}
we have
\begin{align*}
    \Phi(t)&=\exp{\left(\int_{0}^{t}\frac{d}{ds}\Psi(s)\, ds\right) }=\exp{\left(\Psi(t)-\Psi(0)\right)},
\end{align*}
where
\begin{equation*}
    \Psi (s) := \int_{0}^{Q(x_{\lambda_{2}}(s),s)}\frac{1}{\lambda_2-\lambda_1}\frac{\partial\lambda_{2}}{\partial w_2}\left(P_0(x_0),v\right) \,dv.
\end{equation*}
Since $\frac{1}{\lambda_2-\lambda_1}\frac{\partial\lambda_{2}}{\partial w_2}$ is smooth and  $Q\left(\mathbb{R} \times [0,\infty)\right)$ is compact as we saw in the proof of Lemma \ref{lma: dlt_0}, there exist two constants $m_{\Phi}$, $M_{\Phi}$ such that 
\begin{equation}\label{ine: Phi}
    0<m_{\Phi}\le \Phi(t) \le M_{\Phi} \quad \text{for all $t\ge0$.}
\end{equation}
Thus, multiplying $\tilde{P}_{\lambda_2}$ by $\Phi$ and differentiating it with respect to time, we obtain the following ODE:
\begin{equation*}
    \frac{d}{dt}\left(\tilde{P}_{\lambda_2}\Phi\right)= -\frac{\partial\lambda_{2}}{\partial w_1}\tilde{P}_{\lambda_2}^2\Phi,
\end{equation*}
so that Lemma \ref{lma: dlt_0} and \eqref{ine: Phi} yield
\begin{equation*}
    \frac{d}{dt}\left(\tilde{P}_{\lambda_2}\Phi\right) \le -\frac{\delta_0}{M_{\Phi}}  \left(\tilde{P}_{\lambda_2}\Phi\right)^2.
\end{equation*}
Since $\tilde{P}_{\lambda_2}(0)=\partial_{x}P_{0}(x_0)<0$ by Lemma \ref{lma: x_0}, there exists a finite time $T^*>0$ such that
\begin{equation}\label{contra-1}
    \tilde{P}_{\lambda_2}(t) \le \frac{M_{\Phi}\tilde{P}_{\lambda_2}(0)}{M_{\Phi}+\tilde{P}_{\lambda_2}(0)\delta_0 t} \rightarrow -\infty \quad as \quad t\rightarrow T^*.
\end{equation}
On the other hand, we claim that there exist two constants $m_{f_1}$, $M_{f_1}$ such that 
\begin{equation}\label{ine: f_1}
    0<m_{f_1}\le f_{1}(\rho(x,t),q(x,t)) \le M_{f_1} \quad for \quad (x,t) \in \mathbb{R}\times[0,\infty).
\end{equation}
Indeed, $f_{1}(\rho(x,t),q(x,t))=(f_1\circ w^{-1})(P(x,t),Q(x,t))$, so that \eqref{ine: f_1} follows from the smoothness of $f_1$ and the compactness of $P\left(\mathbb{R} \times [0,\infty)\right)$ and $Q\left(\mathbb{R} \times [0,\infty)\right)$ as in the proof of Lemma \ref{lma: dlt_0}.
Hence \eqref{f-r1-asm-1} and \eqref{ine: f_1} give
\begin{align*}
    \normif{\tilde{P}(T^*)} &= \normif{\partial_{x}P(T^*)} \nonumber\\
    &=\normif{f_{1}(\rho(T^*),q(T^*)) \left(\partial_x \rho(T^*) + \partial_{x}q(T^*) \frac{-q(T^*)+\sqrt{q^2(T^*)+4\rho(T^*)}}{2}\right)} < \infty,
\end{align*}
which is a contradiction to \eqref{contra-1}. This shows \eqref{fb-1-1}.

Now we prove the solution $(\rho,c)$ satisfies \eqref{fb-1-but bounded}.
Note that we have
\begin{equation*}
    {\left\Vert \rho(t)  \right\Vert}_{W^{1,\infty}(\mathbb{R})}  +{\left\Vert c(t)  \right\Vert}_{W^{2,\infty}(\mathbb{R})} < \infty
\end{equation*}
on $[0,T^*)$.
This ensures $(\rho,c)$ is smooth and unique, and has the positive lower bounds on $[0,T^*)$ by Remark \ref{rmk: reg.-1} and Remark \ref{rmk: reg.-2}. Then we again define
$(P,Q):\mathbb{R} \times [0,T^*) \rightarrow w(\Omega)$ by
\begin{equation*}\label{eq: PQ}
    P(x,t)=w_1(\rho(x,t),q(x,t)), \quad Q(x,t)=w_2(\rho(x,t),q(x,t)).
\end{equation*}
(In this time, we defined $(P,Q)$ on time interval $[0,T^*)$.)
Using the characteristic defined by 
each $\lambda_{i}(\rho,q)$ $(i=1,2)$, we can check from \eqref{pde: PQ} that 
\begin{equation*}
P\left(\mathbb{R} \times [0,T^*)\right)=P_{0}\left(\mathbb{R}\right)=w_1\left(\rho_{0}\left(\mathbb{R}\right),q_{0}\left(\mathbb{R}\right)\right)=w_1\left(\rho_{0}\left(\mathbb{R}\right),\frac{\partial_{x}c_0}{c_0}\left(\mathbb{R}\right)\right),
\end{equation*}
and
\begin{equation*}
Q\left(\mathbb{R} \times [0,T^*)\right)=Q_{0}\left(\mathbb{R}\right)=w_2\left(\rho_{0}\left(\mathbb{R}\right),q_{0}\left(\mathbb{R}\right)\right)=w_2\left(\rho_{0}\left(\mathbb{R}\right),\frac{\partial_{x}c_0}{c_0}\left(\mathbb{R}\right)\right),
\end{equation*}
so that $P\left(\mathbb{R} \times [0,T^*)\right)$ and $Q\left(\mathbb{R} \times [0,T^*)\right)$ are compact. Hence, noticing 
$(\rho,q)=w^{-1}\circ(P,Q)$ and $w:\Omega \rightarrow w(\Omega)$ is a global diffeomorphism,
$\rho\left(\mathbb{R} \times [0,T^*)\right) \times q \left(\mathbb{R} \times [0,T^*)\right)$ is also compact. This implies that
\begin{equation*}
    \sup_{t\in[0,T^*)}\left({\left\Vert \rho(t) \right\Vert}_{L^{\infty}(\mathbb{R})}  +{\left\Vert q(t)  \right\Vert}_{L^{\infty}(\mathbb{R})}\right) < \infty.
\end{equation*}
From $\partial_t c = -\mu c\rho,$
we have $\sup_{t\in[0,T^*)}\norm{c(t)}_{L^{\infty}(\mathbb{R})}<\infty.$
Recalling $q=\frac{\partial_x c}{c}$, we obtain
$\sup_{t\in[0,T^*)} \norm{\partial_x c (t)}_{L^{\infty}(\mathbb{R})}<\infty.$
This completes the proof of Theorem \ref{thm: blow-1}. $\Box$

\section{Finite-time blow-up in $\mathbb{R}^d$}\label{pf: fb-d}
In this section, we prove Theorem \ref{thm: blow-d} and Corollary \ref{cor: blow-d}. We firstly derive the finite speed of propagation of \eqref{H-KS}.
\begin{lemma}[Finite speed of propagation]\label{lemma: fps}
Let $(\rho_i,c_i)$ $(i=1,2)$ be two solutions of \eqref{H-KS} corresponding to initial data
$(\rho_{i,0},c_{i,0})$ satisfying \eqref{asm-1} such that
\begin{equation*}
    A:=1+\sup_{t\in[0,T^*)}\left(\norm{\rho_2(t)}_{L^{\infty}(\mathbb{R}^d)} + \norm{\nabla \log c_1(t)}_{L^{\infty}(\mathbb{R}^d)}\right) < \infty
\end{equation*}
and
\begin{equation*}
    \norm{\rho_i(t)}_{W^{1,\infty}(\mathbb{R}^d)} + \norm{c_i(t)}_{W^{2,\infty}(\mathbb{R}^d)} < \infty
\qquad (i=1,2)
\end{equation*}
for $t\in[0,T^*)$.
Suppose that $(\rho_{1,0},c_{1,0})=(\rho_{2,0},c_{2,0})$ in a ball $B_{A,T^*}(\bar{x})$, where
\begin{equation*}
    B_{A,T^*}(\bar{x}):=\left\{x\,:\,|x-\bar{x}|\le 6AdT^*\right\}.
\end{equation*}
Then
$(\rho_1,c_1)=(\rho_2,c_2)$ within a cone $K_{A,T^*}(\bar{x})$, where
\begin{equation*}
    K_{A,T^*}(\bar{x}):=\left\{(x,t) \,: \, 0\le t < T^*,\,|x-\bar{x}|\le 6
Ad(T^*-t)\right\}.
\end{equation*}
\end{lemma}
\begin{proof}
We use the notation \eqref{def: difference} in the proof.
Note that Remark \ref{rmk: reg.-1} and Remark \ref{rmk: reg.-2} imply our assumptions are sufficient to guarantee that the solution is smooth, unique, and has positive lower bounds on $[0,T^*)$.
We define a local energy
\begin{equation*}
    E(t):=\frac{1}{2}\left(\norm{\tilde{\rho}(t)}^2_{L^2\left(B_{A,T^*-t}(\bar{x})\right)} 
    +\norm{\sqrt{\rho_2}\tilde{q}(t)}^2_{L^2\left(B_{A,T^*-t}(\bar{x})\right)}\right),
\end{equation*}
where
\begin{equation*}
    B_{A,T^*-t}(\bar{x}):=\left\{x\,:\,|x-\bar{x}|\le 6Ad(T^*-t)\right\}.
\end{equation*}
Recalling \eqref{eq: tilde rho} and \eqref{eq: tilde q},
we compute
\begin{equation*}
    \frac{d}{dt} E(t) = \RN{1}+\RN{2}+\RN{3}+\RN{4} -3Ad \left(\int_{\partial B_{A,T^*-t}(\bar{x})} \Tilde{\rho}^2 + \int_{\partial B_{A,T^*-t}(\bar{x})} \rho_2|\Tilde{q}|^2 \right)
\end{equation*}
with
\begin{align*}
    \RN{1}&=\int_{B_{A,T^*-t}(\bar{x})} \tilde{\rho} \nabla\tilde{\rho}\cdot q_1+\int_{B_{A,T^*-t}(\bar{x})} \tilde{\rho}^2 \nabla \cdot q_1, &
    \quad \RN{2}&=\int_{B_{A,T^*-t}(\bar{x})}  \tilde{\rho}\nabla\rho_2\cdot \tilde{q} + \int_{B_{A,T^*-t}(\bar{x})} \tilde{\rho}\rho_2 \nabla \cdot \tilde{q},\\
    \RN{3}&=\frac{1}{2}\int_{B_{A,T^*-t}(\bar{x})}  |\tilde{q}|^2\nabla \cdot (\rho_2 q_2), &
    \quad \RN{4}&=\int_{B_{A,T^*-t}(\bar{x})} \rho_2 \tilde{q} \cdot \nabla \tilde{\rho}.
\end{align*}
Using the integration by parts on the first term of $\RN{1}$, we have
\begin{equation*}
\begin{split}
    \RN{1}&=\frac{1}{2}\int_{B_{A,T^*-t}(\bar{x})} \nabla\left(\tilde{\rho}^2\right) \cdot q_1
    +\int_{B_{A,T^*-t}(\bar{x})} \tilde{\rho}^2 \nabla \cdot q_1 \\
    &= -\frac{1}{2} \int_{B_{A,T^*-t}(\bar{x})} \tilde{\rho}^2 \nabla \cdot q_1 
    +\frac{1}{2} \int_{\partial B_{A,T^*-t}(\bar{x})} \tilde{\rho}^2 \left(\sum_{k=1}^d (q_1)_k\right)
    +\int_{B_{A,T^*-t}(\bar{x})} \tilde{\rho}^2 \nabla \cdot q_1 \\
    &\le  \frac{Ad}{2} \int_{\partial B_{A,T^*-t}(\bar{x})} \tilde{\rho}^2 +  \norm{q_1}_{W^{1,\infty}}\norm{\tilde{\rho}}^2_{L^2\left(B_{A,T^*-t}(\bar{x})\right)},
\end{split}
\end{equation*}
where in the second line, we used the notation $q_1=((q_1)_1,\cdots,(q_1)_d)$.
After integrating by parts, we employ the Cauchy inequality to obtain
\begin{equation*}
    \RN{2}+\RN{4}=\int_{\partial B_{A,T^*-t}(\bar{x})} \rho_2 \left(\sum_{i=1}^d (\Tilde{q})_i\right) \tilde{\rho} \le \int_{\partial B_{A,T^*-t}(\bar{x})} \rho_2\Tilde{\rho}^2 + \rho_2\left(\sum_{k=1}^d (\Tilde{q})_k\right)^2
    \le 2Ad\int_{\partial B_{A,T^*-t}(\bar{x})} \Tilde{\rho}^2 + \rho_2 |\Tilde{q}|^2.
\end{equation*}
For $\RN{3}$, we have
\begin{equation*}
    \RN{3} \le \normif{\rho_2^{-1}}\norm{\sqrt{\rho_2}\tilde q}^2_{L^2\left(B_{A,T^*-t}(\bar{x})\right)}\norm{\rho_2}_{W^{1,\infty}}\norm{q_2}_{W^{1,\infty}}.
\end{equation*}
Combining all, we arrive at
\begin{equation*}
    \frac{d}{dt}E(t) \le \left( \norm{q_1}_{W^{1,\infty}} + \normif{\rho_2^{-1}}\norm{\rho_2}_{W^{1,\infty}}\norm{q_2}_{W^{1,\infty}} \right) E(t)
\end{equation*}
for $t\in[0,T^*)$. But since $E(0)=0$, $E(t)=0$ for $t\in[0,T^*)$.
To show $\tilde c=0$ in $K_{A,T^*}(\bar{x})$, we consider the equation of $\Tilde{c}$:
\begin{equation*}
    \partial_t \Tilde{c} = -\mu\left(\Tilde{c}\rho_1 + c_2\Tilde{\rho}\right)=-\mu\Tilde{c}\rho_1,
\end{equation*}
where the second equality follows from $\Tilde{\rho}=0$ in $K_{A,T^*}(\bar{x})$.
With the same argument as above, we can show that $\tilde c(t)=0$ for $t\in[0,T^*)$.
\end{proof}

\begin{proof}[Proof of Theorem \ref{thm: blow-d}]
We divide our proof into $\mathbb{R}$ and $\mathbb{R}^d$ $(d\ge 2)$ cases.

\medskip

\textbf{Case 1. $\mathbb{R}$ :} 

\noindent We consider initial data $(\rho_0^1,c_0^1)$ defined by 
\begin{equation}\label{initial data 1D}
        \rho_0^1 := \Bar{\rho} +\psi \quad \text{and} \quad c_0^1 := \Bar{c} + \psi,
    \end{equation}
where $\bar{\rho},\,\bar{c}>0$ are constants to be determined later, and $\psi : \mathbb{R}\rightarrow[0,1]$ is a smooth bump function such that
    \begin{align*}
 \psi(x)=
     \begin{cases}
      1 \quad (|x| \le 1), \\
      0 \quad (|x|\ge 2).
     \end{cases}
 \end{align*}
 Then as we have seen in Remark \ref{rmk. bump function}, $(\rho_0,c_0)$ satisfies all of the conditions given in Theorem \ref{thm: blow-1}, so that the corresponding solution $(\rho^1,c^1)$ blows up at some time $T^*$. 
 The proof consists of two steps. In Step 1, we prove that the blow-up actually occurs on a bounded interval in $\mathbb{R}$. Then in Step 2, we rescale the solution to make it blow up within any given interval.

 $\textbf{Step 1.}$
 To get information for the region where the blow-up occurs, we make use of Lemma \ref{lemma: fps}. Noticing $(\rho_0^1,c_0^1)=(\Bar{\rho},\bar{c})$ on $\mathbb{R}\backslash (-2,2)$ 
 and the corresponding solution to $(\Bar{\rho},\bar{c})$ is $(\Bar{\rho},\bar{c}e^{-\mu \bar{\rho}t})$,
Lemma \ref{lemma: fps} implies $(\rho^1,c^1)=(\bar{\rho},\bar{c}e^{-\mu \bar{\rho}t})$ in
\begin{equation*}
    \left\{(x,t)\in (\mathbb{R}\backslash (-2,2))\times[0,T^*) \,:\, 
    x\le -2-6A^1t \;\, \text{or} \;\, x\ge 2+ 6A^1t\right\},
\end{equation*}
where
\begin{equation}\label{def: A1}
    A^1=1+\bar{\rho}+\sup_{t\in[0,T^*)}\norm{\nabla \log c^1(t)}_{L^{\infty}(\mathbb{R})}.
\end{equation}
(Note that $\sup_{t\in[0,T^*)}\norm{\nabla \log c^1(t)}_{L^{\infty}(\mathbb{R})}$ is clearly bounded due to \eqref{fb-1-but bounded} in Theorem \ref{thm: blow-1}.) Hence, we can conclude that our blow-up of $(\rho^1,c^1)$ occurs on $(-2-6A^1T^*,2+ 6A^1T^*)\times \left\{t=T^* \right\}$.
\begin{remark}\label{rmk: Step 1}
    In this Step 1, we have not given any restriction to $\Bar{\rho}$ and $\bar{c}$. In other words, for any constants $\Bar{\rho}, \Bar{c}>0$, we can show the blow-up region is bounded. This fact will be used in the proof of Corollary \ref{cor: blow-d}.
\end{remark}

$\textbf{Step 2.}$
It suffices to consider the case when the center of the interval is the origin since \eqref{H-KS} is translation invariant.
Let $(-r,r)$ be any interval. Our aim is to rescale $(\rho^1,c^1)$ so that the blowup occurs within $(-r,r)$. Since $(\rho^1,c^1)$ is a solution of \eqref{H-KS}, we can check that $(\rho_a^1,c_a^1)$ is also a solution of \eqref{H-KS}, where
\begin{equation*}
    \rho_a^1(x,t):= a^2\rho^1(ax,a^2t),\qquad c_a^1(x,t):=c^1(ax,a^2t) \qquad(a>0).
\end{equation*}
Then since 
\begin{equation*}
    \rho_a^1(x,0)= a^2(\bar{\rho}+\psi(ax)),\qquad c_a^1(x,0)=\bar{c}+\psi(ax),
\end{equation*}
and $(\rho_a^1(x,0),c_a^1(x,0))=(a^2\Bar{\rho},\bar{c})$ on $\mathbb{R}\backslash \left(-\frac{2}{a},\frac{2}{a}\right)$, Lemma \ref{lemma: fps} implies $(\rho_a^1,c_a^1)=(a^2\bar{\rho},\bar{c}e^{-\mu a^2\bar{\rho}t})$ in
\begin{equation*}
    \left\{(x,t)\in \left(\mathbb{R}\backslash \left(-\frac{2}{a},\frac{2}{a}\right)\right)\times\left[0,\frac{T^*}{a^2}\right) \,:\, 
    x\le -\frac{2}{a}-6A^1_at \;\, \text{or} \;\, x\ge \frac{2}{a}+ 6A^1_at\right\},
\end{equation*}
where
\begin{equation*}
    A^1_a:=a^2(1+\bar{\rho})+\sup_{t\in[0,\,\frac{T^*}{a^2})}\norm{\nabla \log c_a^1(t)}_{L^{\infty}(\mathbb{R})}
    =a^2(1+\bar{\rho})+a\sup_{t\in[0,T^*)}\norm{\nabla \log c^1(t)}_{L^{\infty}(\mathbb{R})}.
\end{equation*}
Hence, $(\rho_a^1,c_a^1)$ blows up on $\left(-\frac{2}{a}-\frac{6A^1_aT^*}{a^2},\frac{2}{a}+ \frac{6A^1_aT^*}{a^2}\right)\times \left\{t=\frac{T^*}{a^2} \right\}$.
Note that
\begin{equation*}
    \frac{2}{a}+ \frac{6A^1_aT^*}{a^2}= \frac{2}{a} + \bar{\rho} + \frac{6\sup_{t\in[0,T^*)}\norm{\nabla \log c^1(t)}_{L^{\infty}(\mathbb{R})}T^*}{a}.
\end{equation*}
Thus, if we choose sufficiently small $\bar{\rho}$ and large $a$, we can make the blow-up occur within $(-r,r)\in \mathbb{R}$.

\medskip

\textbf{Case 2. $\mathbb{R}^d$ $(d\ge2)$ :}

\noindent We shall use the same notation as the above Case 1.
We choose initial data $(\rho_{0},c_{0})$ defined by
\begin{equation}\label{eq: thm3-1}
    \rho_{0}(x):=\psi(x_1)\prod_{k=2}^{d}\psi(\delta x_k)+\bar{\rho}, \quad c_0(x):=\psi(x_1) \prod_{k=2}^{d}\psi(\delta x_k)+\bar{c},
\end{equation}
where $\bar{\rho},\,\bar{c}>0$ are any constants and $\delta>0$ is a small constant to be determined later. 
Note that $(\rho_0,c_0)$ satisfies \eqref{asm-1} and $(\nabla\rho_0,\nabla c_0)\in (H^{\infty} \times H^{\infty})(\mathbb{R}^d)$. 
We aim to show that the corresponding solution $(\rho,c)$ undergoes
\begin{equation*}
    \lim_{t\rightarrow T^*} \left( {\left\Vert \rho(t)  \right\Vert}_{W^{1,\infty}(\mathbb{R}^d)}  +{\left\Vert c(t)  \right\Vert}_{W^{2,\infty}(\mathbb{R}^d)} \right) =\infty
\end{equation*}
for some $T^*>0$.
Suppose, toward the contradiction, that $(\rho,c)$ satisfies
\begin{equation} \label{f-rd-asm-1}
    {\left\Vert \rho(t)  \right\Vert}_{W^{1,\infty}(\mathbb{R}^d)}  +{\left\Vert c(t)  \right\Vert}_{W^{2,\infty}(\mathbb{R}^d)} < \infty \quad \text{on} \;\, [0,T]
\end{equation} 
for any $T>0$.
Recall that \eqref{f-rd-asm-1} is sufficient to guarantee that $(\rho,c)$ has the positive lower bounds and is unique and smooth on $[0,T]$ by Remark \ref{rmk: reg.-1} and Remark \ref{rmk: reg.-2}.

On the other hand, we consider another initial data $(\rho_0^*,c_0^*):\mathbb{R}^d\rightarrow \Omega$ defined by
 \begin{equation*}\label{eq: thm3-2}
    \rho_0^*(x):=\rho_0^1(x_1), \quad c_0^*(x):=c_0^1(x_1).
\end{equation*}
Recalling \eqref{initial data 1D}, we can notice $(\rho_0,c_0)=(\rho_0^*,c_0^*)$ in  $\left\{x:|x_k|\le \delta^{-1},\,k=2,\cdots,d\right\}$.
Moreover, we observe that the unique solution $(\rho^*,c^*):\mathbb{R}^d\rightarrow \Omega$ corresponding to $(\rho_0^*,c_0^*)$ is
\begin{equation*}\label{eq: thm3-3}
    (\rho^*(x,t),c^*(x,t))=(\rho^1(x_1,t), c^1(x_1,t)).
\end{equation*}
Indeed, we can check that $(\rho^*,c^*)$ solves \eqref{H-KS} in $\mathbb{R}^d$ and the uniqueness is ensured by Remark \ref{rmk: reg.-2}. Hence, recalling Step 1 in the proof of Case 1, $(\rho^*,c^*)$ blows up in $\left\{|x|\le 2+6A^1T^*\right\}\times \left\{t=T^*\right\}$. 
Our strategy is to compare $(\rho,c)$ with $(\rho^*,c^*)$ by employing Lemma \ref{lemma: fps}.
Define $A$ by
\begin{equation*}
    A:=1+\sup_{t\in[0,\,T^*)}\left(\norm{\rho(t)}_{L^{\infty}(\mathbb{R}^d)} + \norm{\nabla \log c^*(t)}_{L^{\infty}(\mathbb{R}^d)}\right).
\end{equation*}
Note that $\sup_{t\in[0,\,T^*)}\norm{\rho(t)}_{L^{\infty}(\mathbb{R}^d)}<\infty$
by our assumption \eqref{f-rd-asm-1}, and
$\sup_{t\in[0,\,T^*)}\norm{\nabla \log c^*(t)}_{L^{\infty}(\mathbb{R}^d)}<\infty$ 
by \eqref{fb-1-but bounded}.
Then Lemma \ref{lemma: fps} asserts that if $(\rho_0,c_0)=(\rho_0^*,c_0^*)$ in $\left\{x:|x|\le 2+6A^1T^*+ 6AdT^*\right\}$, 
then
$(\rho,c)=(\rho^*,c^*)$ within $\left\{(x,t) :  0\le t < T^*,\,|x|\le 2+6A^1T^*+ 6
Ad(T^* -t)\right\}$, where $A^1$ is the constant defined in \eqref{def: A1}.

Since $(\rho_0,c_0)=(\rho_0^*,c_0^*)$ in $\left\{x:|x_k|\le \delta^{-1},\,k=2,\cdots,d\right\}$, 
if we choose sufficiently small $\delta$ satisfying 
\begin{equation*}
    2+6A^1T^*+ 6AdT^* \le \delta^{-1},
\end{equation*}
Lemma \ref{lemma: fps} guarantees $(\rho,c)=(\rho^*,c^*)$ within $\left\{(x,t) :  0\le t < T^*,\,|x|\le 2+6A^1T^*+ 6
Ad(T^*-t)\right\}$. But since $(\rho^*,c^*)$ blows up in $\left\{|x|\le 2+6A^1T^*\right\}\times \left\{t=T^*\right\}$, this is a contradiction to \eqref{f-rd-asm-1}.
This completes the proof of Theorem \ref{thm: blow-d}.
\end{proof}

\begin{proof}[Proof of Corollary \ref{cor: blow-d}]
We construct the desired initial data 
by slightly changing \eqref{eq: thm3-1}: we redefine
$(\rho_0,c_0)$ as follows:
\begin{equation*}
    \rho_{0}(x):=\delta^N\psi(x_1) \prod_{k=2}^{d}\psi(\delta x_k)+\bar{\rho}, \quad c_0(x):=\delta^N\psi(x_1)\prod_{k=2}^{d}\psi(\delta x_k)+\bar{c}.
\end{equation*}
for large $N$. Here $\bar{\rho},\,\bar{c}>0$ are any constants.
Recalling our proof of the above Case 2, we needed the smallness of $\delta$. Thus we can control $\delta,N$ so that our new $(\rho_0,c_0)$ satisfies \eqref{asm: coro}, and then proceed as we did in the above proof of Case 2 to obtain the finite blow-up result. 
We omit the details.
\end{proof}

\medskip

\subsection*{Acknowledgments}{The author was supported by the Samsung Science and Technology Foundation under Project Number SSTF-BA2002-04. He thanks In-Jee Jeong for educational discussions and comments.}

\bibliographystyle{amsplain}
\bibliography{Keller-Segel}

\end{document}